\documentclass[10pt, twoside, reqno]{amsart}

\usepackage{gsstyles3}
\usepackage{gscommands3}

\usepackage[headsep=30pt,footskip=25pt, left = 1in,  right = 1in, bottom=1.25in,
top=1.3in]{geometry}

\title[Tensor t-structures, perversity functions
and weight structures]{Tensor t-structures, perversity functions
and weight structures}
\date{} 
\dedicatory{}

\author[Gopinath Sahoo]{Gopinath Sahoo}

\address{Tata Institute of Fundamental Research,  Mumbai,  India.}
\curraddr{}
\email{gsahoo@math.tifr.res.in}
\thanks{}


\begin{document}

\begin{abstract}
 We introduce the notion of tensor t-structures on the bounded derived categories of schemes. 
For a Noetherian scheme $X$ admitting a dualizing complex,  Bezrukavnikov-Deligne,  and then 
independently Gabber and Kashiwara have shown that given a monotone 
comonotone perversity function on $X$ one can construct a t-structure on $\deriveb X$.  We 
show that such t-structures are tensor t-structures and conversely every tensor t-structure on 
$\deriveb X$ arises in this way.  We achieve this by first characterising tensor t-structures in 
terms of Thomason-Cousin filtrations which generalises earlier results of Alonso,  Jeremías and Saorín, from 
Noetherian rings to schemes.  We also show that for a smooth projective curve $C$, the derived 
category $\deriveb C$ has no non-trivial tensor weight structures, this extends our earlier 
result on  the projective line to higher genus curves.

\end{abstract}

\maketitle

\tableofcontents
%======================Write here====================================

\section{Introduction}

The classification theorem of thick subcategories of finite spectra in the stable homotopy 
theory \cite[Theorem 7]{Hop87} sparked the interest to study similar problems in the 
algebraic setting of derived categories.  Hopkins,  \cite[Theorem 11]{Hop87}, proved that the 
thick subcategories of $\perfect R$, for a commutative ring $R$, are completely determined by 
the specialisation-closed subsets of $\Spec R$.  For Noetherian rings,  in \cite{Nee92},
Neeman classified localizing and smashing subcategories of the unbounded derived category 
$\mathbf{D}(R)$.  Thomason in \cite[Theorem 3.15]{Thomason} classified the 
thick tensor ideals of $\perfect X$, generalizing the theorem of Hopkins from  
rings to schemes.  Ever since the work of Hopkins, Neeman and Thomason there has 
been a great interest to understand various subcategories of triangulated categories arising 
both in algebraic geometry and representation theory.

The notion of t-structures on a triangulated category was introduced  in \cite{BBD} by Belinson 
et al. in their study of perverse sheaves.   In our earlier work \cite{DS23}, motivated  by the 
notion of rigid aisles in \cite[\S 5]{AJS03} we defined the notion of tensor t-structures on a 
triangulated category,  in particular  on $\derive X$  where $X$ is a Noetherian scheme.  
These t-structures restrict well to the open subsets.  Using the local nature of these 
t-structures, we were able to generalize earlier results from Noetherian rings to schemes.  
For instance, we have shown in \cite[Theorem 5.11]{DS23} that the compactly generated tensor 
t-structures on $\derive X$ are in one-to-one correspondence with the Thomason filtrations on 
$X$ which generalizes the result of Alonso,  Jeremías and Saorín \cite[Theorem 3.11]{AJS10}.

One of the two goals of this article is to study similar notions of tensor t-structures on the 
bounded 
derived category $\deriveb X$ and connect  it to some of the known results of \cite{AJS10}, 
\cite{Bez10}, \cite{Gabber} and \cite{Kashiwara}.  But in this generality the 
derived tensor product on $\derive X$ does not 
restrict to the bounded derived category $\deriveb X$ unless $X$ is regular.  So we can not 
define it
using the aisle of the standard t-structure as it is done in the unbounded case of $\derive X$.   
However,  
$\perfect {X}$ has a well defined action on $\deriveb X$.  Using this we have defined tensor 
t-structures on the derived category $\deriveb X$, see Definition \ref{tensor definition} and  
Remark \ref{key remark}.  Once we have the right definition we could prove that
every tensor t-structure on the bounded derived category is completely determined by its
support data.  Suppose $\varphi$ is a Thomason-Cousin filtration,  see Definition 
\ref{Thomason-Cousin},  consider the subcategory of $\deriveb X$ associated to $\varphi$ which is defined as:
\[ \cat U _{\varphi} = \{ A \in \deriveb X \mid \Supp (H^i(A)) \subset \varphi (i) \}. \]

Then,  we have the following theorem,  it is a restatement of Theorem \ref{Main theorem}.  And
it generalizes \cite[Theorem 6.9]{AJS10} from Noetherian rings to schemes.

\begin{theorem}
\label{Theorem alpha}
	Let $X$ be a Noetherian scheme admitting a dualizing complex.  The pair of subcategories 
	$(\cat U _\varphi , \cat U _\varphi ^{\perp} )$ is a tensor t-structure 
	on $\deriveb X$.  
	
	Conversely,  any tensor t-structure $(\cat U,  \cat V)$ on $\deriveb X$ arises from a 
	Thomason-Cousin filtration,  that is, there is a unique Thomason-Cousin filtration 
	$\varphi$ such that $\cat U = \cat U _\varphi$ and $\cat V = \cat U _\varphi ^{\perp}$.
\end{theorem}

As an easy consequence of this theorem, we show that every semi-orthogonal decomposition 
on $\deriveb X$ that satisfies the tensor condition is trivial,  see Corollary \ref{tensor SoD}.  In 
particular,  for a connected Noetherian ring admitting a dualizing complex,  the bounded derived 
category $\deriveb R$ is indecomposable.

In the expository article \cite{Bez10},  Bezrukavnikov brought to attention the 
unpublished results of Deligne concerning constructions of t-structures from perversity 
functions. More precisely,  in \cite[Theorem 1]{Bez10},  the author shows that given a monotone 
comonotone perversity 
function on $X$ one can construct a t-structure on $\deriveb X$.  Independently Gabber 
\cite[Theorem 9.1]{Gabber} and then Kashiwara \cite[Theorem 5.9]{Kashiwara} has also 
obtained the same result. 
In Theorem \ref{Perversity} we show such t-structures are tensor t-structures,  and the 
converse is true, that is, every tensor t-structure on 
$\deriveb X$ arises in this way.  This provides a complete picture of this circle of ideas,  and 
connects the very categorical notion of tensor t-structures on $\deriveb X$ to both the
t-structures arising from perversity functions, and t-structures arising from Thomason-Cousin 
filtrations.  As an application of the characterization in terms of perversity functions, in 
Proposition \ref{pullback} we have shown, when the pullback of a tensor t-structure is a tensor 
t-structure.  In particular,  when $X$ is an integral finite type scheme over a field $k$, unlike 
semi-orthogonal decompositions, the tensor t-structures on $\deriveb X$ always restricts 
to closed subschemes.

When this article was at the final stage of preparation we came across the preprint 
\cite{Biswas},  we were aware of their earlier shorter version in which the authors had 
obtained the non-existence of tensor t-structures on perfect complexes of singular Noetherian 
schemes.  In the updated one the authors have now studied the bounded derived categories.
Their \cite[Theorem 3.12]{Biswas} is similar to our Corollary \ref{key corollary},  probably this
is the only common point.  Though they also use our earlier results of \cite{DS23} as we do to 
obtain the above corollary,  the perspective is quite different,  we use our classification 
Theorem \ref{Main theorem}.  Nevertheless,  the notion of tensor t-structure on $\deriveb X$ 
and the connection to perversity functions is completely new in our article.

The second goal of this article is to study the notion of weight structures on the bounded 
derived categories of smooth projective curves.  Weight structures were introduced by 
Bondarko  as a natural counterpart of t-structures in \cite{Bon10}. It has been observed by 
Bondarko that the two notions,  t-structures and weight structures, are connected by interesting 
relations.  In this vein,  S\v{t}ov\'{\i}\v{c}ek and 
Posp\'{\i}\v{s}il have proved for a certain class of triangulated categories,  in particular for 
 the derived category $\mathbf{D}(R)$ of a Noetherian ring, the 
 compactly generated t-structures and compactly generated weight structures are in bijection 
\cite[Theorem 4.10 and Theorem 4.15]{SP16} with each other,  where the bijection goes via
a duality at the perfect complex level; note that the authors call them co-t-structures what we
call weight structures.

Analogous to the notion of tensor t-structures we have introduced the notion of tensor weight structures in \cite{DS23_2} on $\derive X$ and 
have generalized the theorem of S\v{t}ov\'{\i}\v{c}ek and Posp\'{\i}\v{s}il from Noetherian rings 
to schemes.  That is,  when $X$ is a separated Noetherian scheme,  there is a one-to-one 
corresponds between the compactly generated tensor t-structures and compactly 
generated tensor weight structures on $\derive X$.  In sharp contrast to the above fact 
when one restricts to the bounded derived 
category $\deriveb X$, unlike tensor t-structures, tensor weight structures seems to vanish 
altogether.  We had earlier found that in the case of projective 
line, \cite[Proposition 30]{DS23_2}, there are no tensor weight structures.  In the last section of 
this article our further investigation shows that the same is true for all curves of higher genus,
see Theorem \ref{weight structure theorem}.

\begin{comment}
\begin{theorem}[\ref{weight structure theorem}]
	Let $X$ be a smooth projective curve over a field $k$ of genus $g \geq 1$.  There are no 
	non-trivial tensor weight structures on $\deriveb X$.
\end{theorem}

\end{comment}

This is interesting since in the affine case, at least for regular Noetherian rings, we have lots of 
tensor weight structures particularly those coming from the brutal truncations. We have also 
found that in the case of elliptic curves,  not just tensor weight structures,  there
are no weight structures on the bounded derived category other than the trivial ones, see 
Theorem \ref{Elliptic Curve Theorem}.  
Using the techniques developed for the above results, we could also provide a new proof of 
the well know fact \cite[Theorem 1.1]{Okawa} that when $X$ is a smooth projective curve of 
genus greater than equal to one, there are no non-trivial semi-orthogonal decompositions on 
$\deriveb X$,  that is, the bounded derived category is indecomposable in this case,  see
Theorem \ref{Sod theorem}.

\section{The Antecedent}

In this section, we briefly recall the definition of t-structures,  compactly generated
t-structures and tensor t-structures;for more details we refer the reader to our earlier work 
\cite{DS23}.  We freely use the notions of preaisle,  aisle and their connection to t-structure,
for details see \cite[\S 2]{DS23}.  We follow the same notations and symbols as used in loc. cit., 
with one exception that is,  instead of writing $(\cat U, \cat V[1])$ as t-structures we use the 
simpler notation $(\cat U, \cat V)$ in this article.

Let \cat T be a triangulated category and $\cat T^c$  denote the full subcategory of compact
objects.  
\begin{definition}
	A \emph{t-structure} on \cat T is a pair of full subcategories $(\cat U, \cat V)$ satisfying the 
	following properties:
	\begin{itemize}
		\item[t1.] $ \cat U[1] \subset \cat U$ and $  \cat V[-1] \subset \cat V$.
		\item[t2.] $\cat U \perp  \cat V$.
		\item[t3.] For any $T \in \cat T$ there is a distinguished triangle
		\[ U \rar T \rar V \rar  U[1]\]
		where $U \in \cat U$ and $V \in \cat V$.  We call such a triangle 
		\emph{truncation decomposition} of $T$.
	\end{itemize}
	
\end{definition}

In \cite{Sta10}, Stanley has shown that the collection of t-structures on $\mathbf{D}(\Z)$ is a 
proper class, not a set. Therefore,  one can not classify them in terms of the subsets of the 
spectrum.  We restrict our attention to the smaller class - compactly generated ones.  

\begin{definition}
	We say a t-structure $(\cat U ,  \cat V)$ on $\cat T$ is \emph{compactly generated} if there 
	is a set 
	of compact objects $\cat S$ in $\cat T ^c$ such that for any non-zero $U \in \cat U$ if 
	$ \Hom (S, U) = 0$ for all
	$S \in \cat S$ then $U$ is zero.  For the sake of brevity we will call such t-structures 
	\emph{compact t-structures}.  
\end{definition}

In \cite{AJS10}, the authors show that the compact t-structures on $\mathbf{D} (R)$ for a 
Noetherian ring can be characterised by their support data.  We restate it here.  Before that
we recall the definition of Thomason filtrations.

\begin{definition} 
	Let $X$ be a Noetherian scheme.  A \emph{Thomason filtration} on $X$ is a map 
	$\varphi : \Z \rar 2^X $ such that $\varphi $ satisfies the following properties:
	\begin{itemize}
		\item[i.] for any $x$, $y$ $\in X$ such that $y \in \bar{\{x\}}$,  if $x \in \varphi (i)$ then
		$y \in \varphi (i)$. 
		
		\item[ii.] $\varphi (i) \supset \varphi(i+1)$ for all $i \in \Z$.

	\end{itemize}
	
	We say a Thomason filtration is of \emph{finite length} if for 
	$j \ll 0$,  $\varphi(j) = \varphi (j-1)$  and for $i \gg 0$,  $\varphi (i) = \emptyset$.

\end{definition}

\begin{theorem}
\label{AJS}
	Let $R$ be a Noetherian ring. There is a one-to-one correspondence between the 
	 compact t-structures on $\mathbf{D}(R)$ and the Thomason filtrations on
	$\Spec R$.
\end{theorem}

\begin{proof}
	\cite[Theorem 3.11]{AJS10}
\end{proof}

The following example is just a special case of \cite[Remark,  page 328]{AJS10}

\begin{example}

	In the derived category of abelian groups the aisle generated by $\Q$ is not compactly 
	generated.  Indeed,  if it is compactly generated then by the classification of compactly 
	generated t-structures it must be $\mathbf{D} \pre 0 (\Z)$,  However,  $\Z \notin 
	\langle \Q \rangle ^{\leq 0} $.

\end{example}

In our earlier work,  we have generalized Theorem \ref{AJS} to arbitrary Noetherian schemes.  
For the analogous statement to be true we introduced the notion of tensor t-structures, 
motivated  by the notion of rigid aisles in \cite[\S 5]{AJS03}.  We recall it here.

\begin{definition}
	A t-structure $(\cat U, \cat V)$ on $\derive X$ is a tensor t-structure if for any $U \in \cat U$
	and any $A \in \mathbf{D}^{\leq 0} _{qc} (X)$ we have $U \otimes A \in \cat U$. 
\end{definition}

Recall that $ \cat U _{\varphi} = \{ A \in \derive X \mid \Supp (H^i(A)) \subset \varphi (i) \}$ and with the above definition we have shown that:

\begin{theorem}
\label{Earlier Theorem}
	Let $X$ be a Noetherian scheme.  There is a one-to-one correspondence between  compact 
	tensor t-structures on $\derive X$ and Thomason filtrations on $X$. In particular, every 
	compact tensor t-structure on $\derive X$ is of the form 
	$(\cat U _\varphi , \cat U _\varphi ^{\perp} )$ where $\varphi$ 
	is a Thomason filtration on $X$.

\end{theorem}

\begin{proof}
	\cite[Theorem 5.11]{DS23}
\end{proof}

\begin{definition}
We say a t-structure $(\cat U, \cat V)$ is \emph{principal} if it is generated by a single object of 
the derived category,  that is,  $\cat U = \langle E \rangle ^{\leq 0} _{\otimes}$ for some complex 
$E \in \derive X$.
\end{definition}

The following statement about principal compact t-structures easily follows from the methods 
developed in our earlier work.  We state and prove it here.

\begin{proposition}
\label{finite}
	Let $X$ be a Noetherian scheme.  There is a one-to-one correspondence between finite 
	length filtrations of closed subsets of $X$ and principal compact tensor t-structures of 
	$\derive X$.
\end{proposition}

\begin{proof}
	One direction is easy.  Let $(\cat U, \cat V)$ be a t-structure on $\derive X$.  It is principal and 
	compact means $\cat U = \langle E \rangle ^{\leq 0} _{\otimes}$ for some perfect complex 
	$E$.  Recall that the Thomason filtration associated to $\cat U$ which we denote by 
	$\varphi_{\cat U}$ is defined as follows 
	$ \varphi_{\cat U} (n) = \bigcup_{i \geq n} \Supp H ^{i}(E)$.  Now since $E$ is perfect it is 
	clear that $ \varphi_{\cat U} (n) $ is closed, and boundedness of $E$ will imply that the 
	filtration is of finite length.
	
	Conversely,  suppose we are given with a finite length filtration of closed subsets say 
	$\varphi$,  that is, $\varphi (n) $ is a closed subset of $X$ for all $n \in \Z$.  From
	 \cite[Lemma 5.10]{DS23} we have  a perfect complex $E_n$ such that 
	$\Supp H ^n( E_n) = \varphi (n)$ and $\Supp H ^m( E_n) \subset \varphi (n)$ for any $m \in 
	\Z$.  
	Now consider  $E \colonequals \oplus E_n$.  Using \cite[Lemma 5.9]{DS23} we can see
	that  $\cat U _{\varphi} = \langle E \rangle ^{\leq 0} _{\otimes}$.  This proves that the 
	t-structure corresponding to a finite length filtration of closed sets is principal and compact.
	
\end{proof}

\section{Tensor t-structures on the bounded derived categories}

\subsection{Tensor t-structures on $\deriveb X$}

\hspace{1mm}

In this section we define the notion of tensor t-structures on the bounded derived category
$\deriveb X$.  But in this generality the derived tensor product on $\derive X$ does not restrict 
to the bounded derived category $\deriveb X$ unless $X$ is regular.  So we can not define it 
using the aisle of the standard t-structure $\mathbf{D}^ {b,  \leq 0} (X)$.  However,  
$\perfect {X}$ 
has a well defined action on $\deriveb X$.  Using this we define tensor t-structures on the 
derived category $\deriveb X$.  

\begin{definition}
\label{tensor definition}
	We say a preaisle $\cat U$ on $\deriveb X$ is a \tensor preaisle if for any $E \in 
	\perfectzero X$ 
	and any $ F \in \cat U$ the derived tensor product $E \otimes F \in \cat U$.  We say a 
	t-structure $(\cat U, \cat V)$ on $\deriveb X$ is a tensor t-structure if the aisle $\cat U$ is a 
	\tensor preaisle.  
\end{definition}  

\begin{remark}
\label{key remark}
	This definition seems to be more reasonable than an intrinsic definition.  In the case of 
	unbounded derived category $\derive X$, if one defines \tensor preaisle using 
	$\perfectzero X$,  then one can easily show that this new definition is equivalent to our 
	earlier definition using the aisle $\mathbf{D}^{\leq 0} _{qc} (X)$ of the standard t-structure.

\end{remark}

\begin{comment}
\begin{remark}	
	One can similarly define the notion of tensor weight structures using $\perfectzero X$.  
	In \cite{DS23_2} we called such weight structures $\otimes ^c$-weight structures. 
	However,  in view of the above.  This seems to be the appropriate definition than a weaker 
	one.
	 
\end{remark}
\end{comment}

Now, with this definition of tensor t-structures on $\deriveb X$ we can prove that such 
t-structures restricts well to open affine subsets.  The arguments we used to show the 
analogous statements in the unbounded case, see \cite[\S 5.2]{DS23}, seems to work in the 
bounded case with appropriate modifications.  

Let $X$ be a Noetherian scheme and $U$ be an open affine subset of $X$ and $j : U \rar X$ 
denote the open immersion.  Let $\cat U$ be a \tensor preaisle of $\deriveb X$.  We define 
$\cat U|_U \colonequals \langle j^* \cat U \rangle ^{\leq 0}$ - the restriction of 
$\cat U$ to the open affine subset $U$.  Suppose $(\cat U , \cat V)$ is a tensor t-structure on 
$\deriveb X$.  We define $\cat V|_U \colonequals (\cat U|_U)^{\perp}$.

The key important lemma is the following:

\begin{lemma}
\label{Key lemma}
	Let $E$ and $F$ be two complexes in $\deriveb X$.  Suppose there is a non-zero morphism 
	$f \in \Hom (E|_U, F|_U) $ in $\deriveb U$.  Then, there is a perfect complex 
	$K \in \perfectzero X$ such that there is a non-zero map $\prim f \in \Hom (E \otimes K , F)$.
\end{lemma}

\begin{proof}
	By assumption we have
	$\Hom (j^*E, j^*F) \neq 0$.  By tensor-hom adjunction we get 
	
	\begin{align*}
	\Hom (j^*E, j^*F)
	& = \Hom (j^* \CO _X ,  \mathbf{R}  \cat H om ^{\cdot}_{U}(j^*E,  j^*F)) \\
	& = \Hom (j^* \CO _X ,  j^* \mathbf{R}  \cat H om ^{\cdot}_{X}(E,  F)).\\
	\end{align*}
	
	Now if we write $\prim F = \mathbf{R}  \cat H om ^{\cdot}_{X}(E,  F)$ then  $\Hom (j^*E, j^*F) \neq 0$
	implies the group $\Hom (j^* \CO_X ,  j^* \prim F) \neq 0$.  Applying 
	\cite[Lemma 4.15]{DS23} to the preaisle $\mathrm{Perf} ^{\leq N} _{Z} (X)$ where $N=0$ 
	and
	$Z=X$,  we get that there is a perfect complex $K \in \mathrm{Perf} ^{\leq 0} (X)$ such that
	$\Hom (K, \prim F) \neq 0$.  Again by tensor-hom adjunction 
	\begin{align*}
	\Hom (K,  \prim F)
	& = \Hom (K ,  \mathbf{R}  \cat H om ^{\cdot}_{X}(E,  F)) \\
	& = \Hom (K \otimes E,  F) \neq 0,\\
	\end{align*}

\end{proof}

\begin{lemma}
\label{j!}
	Let $\cat U$ be a \tensor preaisle of \deriveb X.   If $F \in (\cat U)^{\perp}$ then $j^*F \in 
	(\cat U|_U)^{\perp}$.  
\end{lemma}

\begin{proof}
	Suppose there is an $F \in (\cat U)^{\perp}$ such that $j^*F \notin
	(\cat U|_U)^{\perp}$ then there is an $E \in \cat U$ such that  
	$\Hom (j^*E, j^*F) \neq 0$.   Now, by the Lemma \ref{Key lemma} we have $K \in 
	\perfectzero X$
	such that $\Hom (E \otimes K , F) \neq 0$.   Since $\cat U$ is a tensor preaisle, 
	$K \otimes E \in \cat U$ and this contradicts the fact that $F \in (\cat U)^{\perp}$.

\end{proof}

 \begin{lemma}
 \label{tilde}
 	Let $X$ be a Noetherian scheme and $U$ be an open subscheme of $X$.  Let $A$ be in 
 	$\deriveb U$ then there exist $\tilde{A}$ in $\deriveb X$ such that $\tilde{A} |_U \cong A$.
 \end{lemma}
 
 \begin{proof}
	It is a special case of \cite[Corollary 2]{Bez10}.
 \end{proof}

\begin{lemma}
\label{localizes to open}
	Let $(\cat U , \cat V)$ be tensor t-structure on \deriveb X.  Then $(\cat U|_U,  \cat V|_U)$ 
	is 
	a t-structure on \deriveb U.  In particular,  for any $A \in \deriveb X$ the triangle 
	 \[ j^*\tau_{\cat U} ^{\leq} A \rar j^*A \rar j^*\tau _{\cat U} ^> A \rar  j^*\tau_{\cat U} 
	^{\leq} A[1],\] is a t-decomposition triangle of $j^*A$ in \deriveb U.
\end{lemma}

\begin{proof}
	To show $(\cat U|_U,  \cat V|_U)$  is a t-structure on $\deriveb U$ it is enough to
	show that for any $T\in \deriveb U$ there is a t-decomposition triangle.
	
	Let $T \in \deriveb U$.  Consider any extension $\tilde{T}$ of $T$ in $\deriveb X$ as in 
	Lemma \ref{tilde}.  Now, since $(\cat U , \cat V)$ is a t-structure on $\deriveb X$ we have
	a t-decomposition in $\deriveb X$ say, 
	\[ \tau_{\cat U} ^{\leq} \tilde{T} \rar \tilde{T} \rar \tau _{\cat U} ^> \tilde{T} \rar  \tau_{\cat U} ^{\leq} \tilde{T}[1].\] 
	
	By the construction of $\cat U |_U$, we have $j^* \tau_{\cat U} ^{\leq} \tilde{T} $ is in  $\cat 
	U |_U
	$, and by our Lemma \ref{j!}, we have $j^* \tau _{\cat U} ^> \tilde{T}$ is in $(\cat U|
	_U)^{\perp}$.
	Therefore, applying $j^*$ to the above triangle give us   
	
	\[ j^*\tau_{\cat U} ^{\leq} \tilde{T} \rar T \rar j^* \tau _{\cat U} ^> \tilde{T} \rar  j^*\tau_{\cat 
	U} ^{\leq} \tilde{T}[1],\] 
	
	which is a t-decomposition triangle for $T$ in $\deriveb U$.

\end{proof}

\begin{lemma}
\label{adjoint restricts}
	Let $\cat U$ be an \tensor aisle of $\deriveb X$ and $\tau _X$ is the truncation functor for 
	$\cat U$.  From the Lemma \ref{localizes to open} we have that $\cat U |_U$ is an aisle 
	of $\deriveb U$.   If $\tau _U$ is the truncation functor for $\cat U |_U$ in $\deriveb U$, 
	then we have the following isomorphism for any $A \in \deriveb X$, 
	 
	 \[ \tau _X (A) |_U \cong \tau _U (A|_U). \]
\end{lemma}

\begin{proof}
	This is immediate from Lemma \ref{localizes to open} and the definition of truncation 
	functor associated to an aisle.
\end{proof}

\begin{remark}
\label{U remark}
	Exact same statement as in Lemma \ref{adjoint restricts} also holds in the 
	unbounded case.  Following \cite[Lemma 5.6]{DS23},  we have,  for any tensor aisle $\cat U$ 
	in $\derive X$, and for $A \in \derive X$, 
	 \[ \tau _X (A) |_U \cong \tau _U (A|_U). \] We will use this in Theorem \ref{Main theorem}.
\end{remark}

\begin{lemma}
\label{Support lemma}
	Let $E$ and $F$ be two complexes in $\deriveb X$.  If $\Supp H^k(E) \subset  
	\cup _{i \geq k} \Supp  ( H^i (F))$ then $E \in \tpreaisle F$,  where $\tpreaisle -$ is taken in 
	$\derive X$.
\end{lemma}

\begin{proof}
	Let us denote the \tensor preaisle $ \tpreaisle F$ by $\cat U ^{F}$.  By
	 \cite[Lemma 5.2]{DS23} and the fact that $F \in \deriveb X$,  we get that the associated 
	 filtration of $\cat U ^F$ is Thomason.  Hence,  by the Theorem \ref{Earlier Theorem} we have
	  $\cat U ^F$ is a \tensor aisle.  Now, since \tensor aisles restrict to aisle on open affine 
	  subsets by \cite[Lemma 5.6]{DS23}, $\cat U ^{F} |_U$ is an aisle in $\derive U$,  for any 
	  affine open subset $U$ of $X$.  We can then apply \cite[Proposition 5.1]{Hrbek} to 
	  $E |_U$ and $F |_U$,  to conclude that  $E |_U \in \cat U ^{F} |_U$.  Note that,  this is 
	  enough to show that $E \in \cat U ^F$.
\end{proof}

 Recall that given  a Thomason filtration $\varphi$ on $X$ one can take the associated 
preaisle $\cat U _{\varphi}$ in the unbounded derived category $\derive X$,  which is a 
compact aisle.  
 
 \begin{proposition}
 \label{phi lemma}
	Let $X$ be a Noetherian scheme and $\cat U$ be a \tensor preaisle on
	$ \mathbf{D}^{b}(X)$.  Then, there is a unique Thomason filtration $\varphi$ on $X$ such 
	that $\cat U = \cat U_{\varphi} \cap  \mathbf{D}^{b}(X)$.
	
\end{proposition}

\begin{proof}
	Let $\cat U$ be a \tensor preaisle of $\deriveb X$.  We can take smallest cocomplete 
	preaisle 
	generated by the objects of $\cat U$ in $\derive X$,  that is,  $\preaisle U$.  Since $\cat U$
	is \tensor preaisle,  $\preaisle U$ is also \tensor preaisle of $\derive X$.  Using the definition 
	of
	graded support \cite{DS23},  and the Lemma 5.2 loc. cit,  we can easily conclude that
	the corresponding filtration of $\preaisle U$ is a Thomason filtration, say $\varphi$ and 
	since $\preaisle U$ is a \tensor preaisle we have $\preaisle U$ = $\cat U _{\varphi}$.
	
	It is clear that $\cat U \subset \cat U _{\varphi} \cap \deriveb X$.  We need to show the 
	reverse inclusion.  Suppose $E \in \cat U _{\varphi} \cap \deriveb X$,  by the construction of 
	$\cat U_{\varphi}$ above we have a collection of objects $B_{\alpha} \in \cat U$ such
	that $\Supp ^{\geq i} E \subset \cup _{i} \Supp ^{\geq i} B_{\alpha} $.  Since $\Supp ^{\geq i}
	E$ is a closed subset of $X$ we can consider the generic points of its irreducible 
	components and the $B_{\alpha}$'s containing the generic points, therefore, there are only 
	finitely many $B_{\alpha}$'s such that $ \Supp ^{\geq i} E \subset \cup _{i} \Supp ^{\geq i} 
	B_{\alpha} $.  Now, we apply the Lemma \ref{Support lemma} to conclude our claim.

\end{proof}

\subsection{Noetherian schemes admitting dualising complexes}

\hspace{1mm}

For a Noetherian ring $R$,  in \cite[Corollary 4.5]{AJS10},  the authors show if a compact 
t-structure $(\cat U,  \cat V)$ on the unbounded derived category $\mathbf{D} (R)$ restrict to a 
t-structure on the bounded derived category $\deriveb R$ then the corresponding Thomason 
filtration satisfies an additional condition, which the authors call weak Cousin condition.  And
in the presence of a dualizing complex in \cite[Theorem 6.9]{AJS10} the authors have shown 
that the filtration being weak-Cousin is necessary and sufficient. 

The goal of the section is to generalise the above theorem for Noetherian rings to schemes as
well as to show that every tensor t-structure on $\deriveb X$ arises in this way.   Let us first 
recall the definition of weak Cousin condition given in \cite[\S 4]{AJS10}.

\begin{definition}
	Let $R$ be a Noetherian ring.  We denote $X \colonequals \Spec R$. 
	Let $\varphi : \Z \rar 2^X$ be a Thomason filtration.  We say that $\varphi$ satisfies the weak
	Cousin condition: if for every $j \in \Z$,  if $\mathfrak{p} \subsetneq \mathfrak{q}$,  with
	$\mathfrak{p}$ maximal under $\mathfrak{q}$ and $\mathfrak{q} \in \varphi (j)$ then
	$\mathfrak{p} \in \varphi (j-1)$.
\end{definition}

Let $X$ be a Noetherian scheme of finite Krull dimension.  We modify the above definition
suitably and call such filtrations Thomason-Cousin instead of weak-Cousin so as to
retain Thomason's name and highlight his seminal contribution in \cite{Thomason}.

\begin{definition}
\label{Thomason-Cousin}
	We say an ordered pair $(y,x)$ of points of $X$ is a \emph{specialization pair} of $X$ if 
	$y \in \bar{\{x\}}$.  Given a specialization pair $(y,x)$ we define the codimension function as 
	$\delta _X (y,x) \colonequals codim( \bar{\{y\}},  \bar{\{x\}}) =\mathrm{dim } \CO_{y, X} - 
	\mathrm{dim } \CO _{x, X} $. 
	
	We say a Thomason filtration is \emph{Thomason-Cousin} if  for any specialization pair 
	$(y,x)$,  if $y \in \varphi(i)$ then $x \in \varphi (i - \delta _X (y,x) )$. 
\end{definition}

Note that a Thomason-Cousin filtration in this case is a finite length filtration,  more precisely
for $i \ll 0$ we have $\varphi (i) = X$ and for $j\gg 0$ we have $\varphi (j)= \emptyset $.

\begin{comment}

The following is just a the restatement of Corollary 4.5 of \cite{AJS10}.
\begin{lemma}
	Let $R$ be a Noetherian ring of finite Krull dimension.  Let $(\cat U, \cat V)$ be 
	a t-structure on $\deriveb R$ and $\varphi$ be the unique Thomason filtration associated
	to it by Proposition \ref{phi lemma}.  Then $\varphi$ is a Thomason-Cousin filtration. 
	
\end{lemma}
\end{comment}

After stating it in the affine case we will extend it to schemes.

\begin{theorem}
\label{main affine theorem}
	Let $R$ be a Noetherian ring admitting  a dualizing complex.  Let $\cat U$ be a preaisle of 
	$\deriveb R$ and $\varphi$ its associated filtration as in \ref{phi lemma}. Then, $\cat U$ is an 
	aisle of $\deriveb R$ if and only if $\varphi$ is Thomason-Cousin.
\end{theorem}

\begin{proof}
	\cite[Theorem 6.9]{AJS10}
\end{proof}

\begin{theorem}
\label{Main theorem}
	Let $X$ be a Noetherian scheme admitting a dualizing complex.  Let $\cat U$ be a \tensor 
	preaisle of $\mathbf{D}^b(X)$ and $\varphi$ its associated filtration as in \ref{phi lemma}.  
	Then,  $\cat U$ is an 
	aisle if and only if $\varphi$ is Thomason-Cousin.
\end{theorem}

\begin{proof}
	First, we will show, if $\cat U$ is an aisle then $\varphi$ is a Thomason-Cousin filtration.  
	
	Let $(y,x)$ be a specialization pair.  Clearly there is an open affine subset $U$ containing 
	both $y$ and $x$.   We can now restrict the aisle $\cat U$ to the open affine subset $U$,
	that is $\cat U |_U$,  which is an aisle on $\deriveb U$ by Lemma \ref{localizes to open}.  The 
	corresponding filtration of $\cat U |_{U}$ is $\varphi |_{U} $.  
	Now since $\cat U |_{U}$ is an aisle,  from Theorem \ref{main affine theorem} we get that 
	$\varphi |_{U} $ is a Thomason-Cousin filtration on $U$.  Since this is true for all 
	pair of specialization points we get that $\varphi$ is Thomason-Cousin.  
	
	Now,  the converse part: we need to show $\cat U$ is an aisle if  $\varphi$ is 
	Thomason-Cousin.  By Proposition \ref{phi lemma} we know that $\cat U =
	\cat U _{\varphi} \cap \deriveb X$,  and $\cat U _{\varphi}$ is an aisle of $\derive X$
	by Theorem \ref{Earlier Theorem}.  Note that to show $\cat U$ is an aisle on $\deriveb X$  it 
	is enough to show that the truncation functor for $\cat U_{\varphi}$, that is,  
	$\tau _{\varphi}  $ restricts to $\deriveb X$.  
	
	In other words,  we need to show for any $A \in \deriveb X$ we have 
	$\tau _{\varphi} (A) \in \deriveb X$.  Indeed, this is the case,  by Remark 
	\ref{U remark} we have $\tau _{\varphi} (A) |_U \cong \tau _{\varphi |_U} (A |_U)$.  
	Now since $\varphi$ is Thomason-Cousin we have $\varphi |_U$ is Thomason-Cousin hence 
	$\tau _{\varphi |_U} (A |_U)$ is in $\deriveb U$ by Theorem \ref{main affine theorem}.  
	Since $U$ is arbitrary we get that $\tau _{\varphi} (A)$ is in $\deriveb X$.  This proves the 
	theorem.

\end{proof}

 \begin{corollary}
 \label{key corollary}
	Let $X$ be a Noetherian scheme admitting a dualizing complex.  A compactly generated
	tensor t-structure $(\cat U , \cat V)$ on $\derive X$ restrict to a t-structure on 
	$\deriveb X
	$ if and only if the corresponding Thomason filtration $\varphi$ satisfies
	 Thomason-Cousin condition.
\end{corollary}

\begin{proof}
	This is just the combination of Theorem \ref{Earlier Theorem},  Proposition \ref{phi lemma}
	 and Theorem \ref{Main theorem}.
\end{proof}

We will give a few corollaries of Theorem \ref{Main theorem} concerning semi-orthogonal 
decompositions.  First, we recall the definition.

\begin{definition}
	A \emph{semi-orthogonal decomposition} on  a triangulated category \cat T is a pair of full 
	subcategories $(\cat A, \cat B)$ satisfying the 
	following properties:
	\begin{itemize}
		\item[s1.] Both \cat A and \cat B are triangulated subcategories.
		\item[s2.] $\cat B \perp  \cat A$.
		\item[s3.] For any $T \in \cat T$ there is a distinguished triangle
		\[ B \rar T \rar A \rar  B[1]\]
		where $B \in \cat B$ and $A \in \cat A$.  
	\end{itemize}
	
\end{definition}

Note that if $(\cat A, \cat B)$ is a semi-orthogonal decomposition on $\cat T$ then 
$(\cat B,\cat A)$ is a t-structure on $\cat T$.  

\begin{definition}
	  We say $(\cat A, \cat B)$ is a \emph{tensor semi-orthogonal decomposition} on 
	  $\deriveb X$ if 
	  $\cat B$ is a \tensor preaisle of $\deriveb X$,  that is,   if for any $E \in \perfectzero X$ and
	  any $ F \in \cat B$ the derived tensor product $E \otimes F \in \cat B$.  
\end{definition}

\begin{corollary}
\label{tensor SoD}
	Let $X$ be a connected Noetherian scheme admitting a dualizing complex.  Every tensor 
	semi-orthogonal decomposition on $\deriveb X$ is trivial. 
\end{corollary}

\begin{proof}
	First note that if $(\cat A, \cat B)$ is a semi-orthogonal decomposition of $\deriveb X$ then
	$(\cat B, \cat A)$ is a t-structure on $\deriveb X$.  Now since $(\cat A, \cat B)$ is tensor
	we have $(\cat B, \cat A)$ is a tensor t-structure on $\deriveb X$.  Then by 
	Theorem \ref{Main theorem} 	there
	is a unique Thomason-Cousin filtration $\varphi$ such that $\cat B = \cat U _{\varphi} \cap
	\deriveb X$.  Since $\cat B$ is triangulated subcategory the filtration must be constant
	and again by our observation that $\varphi (i) = X$ for some $i \ll 0$ it must be the constant
	filtration, that is,  $\varphi (i) = X$ for all $i$,  this implies that $\cat B = \deriveb X$.  

\end{proof}

\begin{corollary}
\label{ring SoD}
	If $R$ is a connected Noetherian ring that admits a dualizing complex then the derived category 
	$\deriveb R$ admits no non-trivial semi-orthogonal decompositions. 
\end{corollary}

\begin{proof}
	It follows from Corollary \ref{tensor SoD} by the noting that every semi-orthogonal 
	decomposition on $\deriveb R $ is tensor.
\end{proof}

\section{Perversity functions and tensor t-structures }

\hspace{1mm}

Let $X$ be a Noetherian scheme admitting a dualizing complex.  In \cite[Theorem 1]{Bez10},  
following Deligne,  Bezrukavnikov has shown that given a monotone comonotone perversity 
function on $X$ one can construct a t-structure on $\deriveb X$.  In this section we show
that such t-structures are tensor t-structures.  And conversely,  every tensor t-structure on 
$\deriveb X$ arises in this way.

\begin{definition}
Any bounded integer valued function on $X$ is called a \emph{perversity} function on $X$.  A  
perversity function is called \emph{monotone} if for any specialization pair $(y,x)$ we have 
$p(y) \geq p(x)$.  

Note that the dimension function $d_X : X \rar \Z$ defined as 
$ x \mapsto \mathrm{dim}(\CO _{x, X})$ is a monotone perversity function.  The \emph{dual} of 
a perversity function $p$ denoted by $\bar{p}$ is defined to be $d_X - p$.  One says a 
perversity function $p$ is comonotone if $\bar{p}$ is monotone.

\end{definition}

Given a monotone perversity $p$ on $X$ one can define a pair of subcategories 
$(\cat U^{p}, \cat V ^{p})$ on $\deriveb X$ as follows:

\[ \cat U ^p \colonequals \{ F \in \deriveb X \mid \text{for any } x,  F_x \in \mathbf{D}^{\leq p(x)}
(\CO _{x, X})\},\]
\[\cat V^p \colonequals \{ G \in \deriveb X \mid \text{for any } x, (\mathbf{R}\Gamma_{\bar{\{x\}} 
}G)_x \in \mathbf{D}^{> p(x)}(\CO _{x, X})  \}.\]

The following is a well known theorem due to Deligne-Bezrukavnikov.

\begin{theorem}
\label{Deligne}
	Let $X$ be a Noetherian scheme admitting a dualizing complex.   Let $p$ be a monotone 
	perversity on $X$ and  $(\cat U ^p, \cat V ^p)$ as defined above.  Then,  
	$(\cat U^p, \cat V^p)$ is a t-structure on $\deriveb X$ if $p$ is comonotone. 
\end{theorem}

\begin{proof}
	\cite[Theorem 1]{Bez10}.
\end{proof}

We will show such t-structures are tensor t-structures and the converse is also true.

\begin{theorem}
\label{Perversity}
	Let $X$ be a Noetherian scheme admitting a dualizing complex.  The pair of subcategories
	$(\cat U ^p, \cat V ^p)$ is a tensor t-structure on $\deriveb X$.
	
	Conversely,  every tensor-t-structure on $\deriveb X$ is of the form $(\cat U^p, \cat V^p)$ 
	 for some monotone comonotone perversity function $p$ on $X$. 
\end{theorem}

\begin{proof}	
	Theorem \ref{Deligne} says it is a t-structure.  We only need to show that it is a tensor 
	t-structure,  that is, if $E \in \perfectzero X$ and $F \in \cat U ^p$ then 
	$E \otimes F \in \cat U ^p$.  From the description of $\cat U ^p$,  we need to show that 
	$(E \otimes F)_x = E_x \otimes F_x \in 	\mathbf{D}^{\leq p(x)}(\CO _{x, X})$ for any 
	$x \in X$.
	
	This we deduce from the following spectral sequence,
	
	\[E ^{i ,j} _2 \colonequals H ^i  (E_x  \otimes H^j (F_x)) \Longrightarrow 
	H ^{i+j} (E_x \otimes F_x). \]
	
	If $i+j > p(x)$ then either $i > 0$ or  $j > p(x)$.  In either case, the $E ^{i, j} _2$
	term vanishes,  since $E_x \in \perfectzero X$ and $F_x \in 
	\mathbf{D}^{\leq p(x)}(\CO _{x, X})$.  Hence $H ^{i+j} (E_x \otimes F_x) = 0$ whenever,
	$i+j > p(x)$.  This proves our claim.
	
	The converse part follows from the below Lemma \ref{phi and p lemma} and our
	Theorem \ref{Theorem alpha} which is just a restatement of Theorem \ref{Main theorem}. 
\end{proof}

\begin{lemma}
\label{phi and p lemma}
The following maps provide a one-to-one correspondence between the set of monotone 
perversity function on $X$ and the set of finite length Thomason filtrations on $X$:

\[ p \longmapsto \varphi,  \hspace*{1cm }\text{ where } \varphi (n) \colonequals \{x\in X \mid p(x) \geq n \} \text{ for } n \in \Z ,\]
and
\[ \varphi \longmapsto p,  \hspace*{1cm }\text{ where } p(x) \colonequals \text{max}\{n \mid x \in \varphi(n)\} \text{ for } x \in X.\]

Moreover,  if a perversity function $p$ corresponds to $\varphi$ under this bijection, then
$\cat U _{\varphi} = \cat U ^p$ and $\cat U _{\varphi} ^\perp = \cat V ^p$.  Further,  $p$ is comonotone if and only if $\varphi$ is Thomason-Cousin.
\end{lemma}

\begin{proof}
	This is an elementary but long check.  Probably this is standard and well known,  but
	we provide a proof for the convenience of the reader.
	
	Suppose $p$ is a perversity function on $X$ and $\varphi$ is defined as above.  Then, it
	is easy to see that  $\varphi$ is a finite length filtration, that is, $\varphi (i+1) \subset \varphi(i)
	$  for all $i \in \Z$ and for some $m > n$, we have $\varphi(n) = X$ and 
	$\varphi(m) = \emptyset$.   
	And,  the converse is also true,  that is, if we start with a filtration $\varphi$ having the above 
	property  
	and define a map $p$ as above then it is a perversity function on $X$.  Now,
	\begin{align*}
	p \text { is monotone}
	& \iff p(y) \geq p(x)  \text{ for all specialization pair } (y,x), \\
	& \iff \text {if } x \in \varphi (i) \text{ then } y \in \varphi (i) \text{ for all specialization pair } (y,x), \\
	& \iff \varphi \text{ is Thomason.}\\
	\end{align*}
	Next, we will show that $\cat U_{\varphi} = \cat U^p$.  
	\begin{align*}
	F \notin \cat U _{\varphi} 
	& \iff \exists i \in \Z,  \text { such that  }  \Supp H^i(F) \nsubseteq \varphi(i), \\
	& \iff \exists  i \in \Z \text { and } x \in X, \text { such that } x \in \Supp H^i(F) \text{ and } x 
	\notin \varphi (i),   \\
	& \iff \exists  i \in \Z \text { and } x \in X,  \text { such that }F_x \notin \mathbf{D}^{\leq i-1}
	(\CO _{x, X}) \text { and } p(x) 
	\leq i-1, \\
	& \iff \exists x \in X,  \text{ such that }F_x \notin \mathbf{D}^{\leq p(x)}(\CO _{x, X}),\\
	& \iff F \notin \cat U ^p \\
	\end{align*}
	
	We have $\cat U_{\varphi} = \cat U^p$,  hence $\cat U _{\varphi} ^\perp = 
	\cat U^p {^\perp}$.  By \cite[Proposition 2]{Bez10} we have $\cat U^p {^\perp} =
	\cat V^p $.  This proves our claim $\cat U _{\varphi} ^\perp =\cat V^p$.  Next,  suppose $p$ is 
	a monotone perversity and $\varphi$ its corresponding Thomason filtration.  Then, we have 
	the following chain of equivalent statements:
\begin{align*}
 p \text{ is comonotone}
& \iff d_X - p \text{ is monotone, }\\
& \iff d_X(y) - p(y) \geq d_X(x) - p(x) \text{ for all specialization pair } (y,x), \\
& \iff p(x) \geq p(y) - (d_X(y) - d_X(x) ) \text{ for all specialization pair } (y,x), \\
& \iff p(x) \geq p(y) - \delta (y, x)  \text{ for all specialization pair } (y,x), \\
& \iff \text {if } y \in \varphi (i) \text { then } x \in \varphi(i -\delta (y,x)) \text{ for all specialization pair } (y,x), \\
& \iff \varphi \text{ is Thomason-Cousin. }
\end{align*}

\end{proof}

Recall that we say a t-structure $(\cat U, \cat V)$ on $\deriveb X$ is principal if the aisle
$\cat U$ is generated by a single object of $\deriveb X$, that is,  there is a complex 
$E \in \deriveb X$  such that $\cat U =  \tpreaisle E$.

\begin{proposition}
	A tensor t-structure on $\deriveb X$ is principal if and only if the corresponding monotone
	comonotone perversity is upper semi-continuous.  
\end{proposition}

\begin{proof}
	Recall that a function $p : X \rar \Z$ is upper semi-continuous if and only if for each 
	$n \in \Z$ the set $\varphi (n) \colonequals \{x\in X \mid p(x) \geq n \}$ is a closed subset of 
	$X$.  The filtration $\{\varphi (n)\} _{n \in \Z}$ is precisely the Thomason-Cousin filtration 
	corresponding to $p$; see Lemma \ref{phi and p lemma}.  Now our claim follows from 
	Proposition \ref{finite} and Proposition \ref{phi lemma}.

\end{proof}

%{\textcolor{red}{\textbf{The following two subsections are yet to completed.}}}

\subsection{Tensor t-structures under the pullback}

\hspace{1mm}

Recall from \cite[\S 2 and \S 3]{DS23} that if $\cat S$ is a collection of objects of $\deriveb X$,  
we denote the smallest \tensor preaisle containing $\cat S$ by $\tpreaisle S$.  Let $f: X \rar Y$ 
be a morphism of schemes.  Suppose $\cat A$ is an aisle of $\deriveb Y$.  We 
define the image of $\cat A$ under the derived pullback functor $\mathbf{L}f^*$ as 
$\langle {\mathbf{L}f^*A \mid A \in \cat A \rangle}^{\leq 0}_{\otimes}$ and denote it by 
$\mathbf{L}f^*\cat A$.
In this section we will answer  the question when $\mathbf{L}f^*\cat A$ is an aisle of 
$\deriveb X$ using the characterization in terms of perversity functions.  First, we need to 
slightly enlarge our objects of study.

\begin{definition}
 	We say a preaisle $\cat U$ of $\deriveb X$ is \emph{bounded }if there are integers
 	 $m \leq n$ such that, \[\mathbf{D}^ {b,  \leq m} (X) \subset \cat U  \subset 
 	 \mathbf{D}^ {b,  \leq n} (X).\]

 \end{definition}
 
 \begin{proposition}
 \label{p lemma}
	Let $X$ be a Noetherian scheme.  Then, there is a one-to-one correspondence between
	the bounded \tensor preaisles of $\deriveb X$ and monotone perversity functions on $X$.
	The bijection is given as follows:
	
	\[ \cat U \longmapsto p _{\cat U},  \hspace*{1cm } \text{ where \hspace*{1mm } }  
	p_{\cat U}(x) \colonequals \text{max} \{ n \mid E_x \in \mathbf{D}^{\leq n}(\CO _{x, X}) 
	\text{ for all } E \in \cat U \}\]
	and,
	
	\[ p \longmapsto \cat U_p,   \hspace*{1cm } \text{ where \hspace*{1mm } }  \cat U ^p 
	\colonequals \{ F \in \deriveb X \mid \text{for any } x,  F_x \in \mathbf{D}^{\leq p(x)}
	(\CO _{x, X})\}.\] 
	
 \end{proposition}
 
 \begin{proof}
 This is a slight extension of Theorem \ref{Perversity}.  The only new thing to check here is 
 boundedness.  Let $p$ be a monotone perversity function.  As $p$ is bounded there are 
 integers $m \leq n
 $ such that $p(x) \in  [m, n]$ for all $x \in X$.  Therefore, $\mathbf{D}^ {b,  \leq m} (X) \subset 
 \cat U _p  \subset \mathbf{D}^ {b,  \leq n} (X)$.  Now if $\cat U$ is a bounded \tensor preaisle
 then it is not hard to show that $p_{\cat U}$ is a monotone perversity function on $X$. The 
 rest follows from the arguments used in the proof of Theorem \ref{Perversity},  Proposition 
 \ref{phi lemma} and Lemma \ref{phi and p lemma} with appropriate modifications.

  \end{proof}

\begin{lemma}
\label{pLf = pf}
	Let $f: X \rar Y$ be a morphism of finite type integral schemes over a field $k$.  Let $\cat A$ 
	 be a bounded \tensor preaisle of $\deriveb Y$.  Then,  $\mathbf{L}f^* \cat A$ is a bounded
	 \tensor preaisle of $\deriveb X$.  Further,  if $p_{\cat A}$ is the monotone perversity 
	 function associated to $\cat A$ as in \ref{p lemma},  then $p_{\cat A} \circ f$ is the 
	 associated monotone perversity function of $\mathbf{L}f^* \cat A$.
\end{lemma}

\begin{proof}
	From the construction of $\mathbf{L}f^* \cat A$, it is clear that it is a \tensor preaisle of
	$\deriveb X$.  To show it is bounded, first we observe that 
	$\mathbf{L}f^*\mathbf{D}^ {b,  \leq i} (Y) = \mathbf{D}^ {b,  \leq i} (X) $
	for any integer $i \in \Z$.  This immediately shows that if $\cat A$ is bounded then 
	$\mathbf{L}f^* \cat A$ is bounded.
	Now from Proposition \ref{p lemma},
	\begin{align*}
	p_{\mathbf{L}f^* \cat A}(x)
	& = \text{max} \{ n \mid F_x \in \mathbf{D}^{\leq n}
	(\CO _{x, X}) \text{ for all } F \in \mathbf{L}f^* \cat A  \},\\
	& = \text{max} \{ n \mid (\mathbf{L}f^*E)_x \in \mathbf{D}^{\leq n}
	(\CO _{x, X}) \text{ for all } E \in \cat A  \}.\\
	\end{align*}
	
	We have $(\mathbf{L}f^*E)_x \cong E_{f(x)} \otimes _{\CO _{f(x), Y}} \CO _{x, X}$,  and
	the following spectral sequence,
	
	\[E ^{i ,j} _2 \colonequals  \mathrm{Tor} ^i _{\CO _{f(x), Y}} ( H^j(E_{f(x)}),  \CO _{x, X} )
	\Longrightarrow 
	H ^{i+j} ((\mathbf{L}f^*E)_x). \]	
	
	Whenever,  $i+j > p(f(x))$,  then either $j > p(f(x))$ or $i >0 $.  In either case, the $E ^{i, j} _2$
	term vanishes.  Therefore,  we get $p_{\mathbf{L}f^* \cat A}(x) \leq p(f(x))$.  
	
	Now,  when $i+j = p(f(x))$,  the only non-vanishing $E ^{i ,j} _2$ term we get is when
	$i =0$ and $j=p(f(x))$.   By definition of $p(f(x))$ there is an object $E \in \cat A$ such that 
	$H ^{p(f(x))}(E_{f(x)}) \neq 0$.  Now,  by Nakayama Lemma we have  
	$ H ^{p(f(x))} ( (\mathbf{L}f^*E)_x) \neq 0$.  This proves that 
	$p_{\mathbf{L}f^* \cat A}(x) = p(f(x))$.

\end{proof}

\begin{lemma}
\label{dx and dy}
	Let $f: X \rar Y$ be a flat morphism or a closed immersion of finite type integral schemes 
	over a field $k$.  If $(a,b)$ is a specialization pair on $X$ then,
	\[ \delta_X(a,b) \geq \delta_Y(f(a), f(b)).\]
\end{lemma}

\begin{proof}
	Note that the continuity of $f$ implies that $f$ takes specialization pairs to specialization
	pairs,  that is, $f(a) \in \bar{\{f(b)\}} $.  Since every open subset containing $f(a)$ contains
	$f(b)$,  we can choose an affine open subset of $Y$ containing the pair $(f(a), f(b))$ say 
	$V$.  Similarly, we can choose an affine open subset 
	containing the pair $(a, b)$ in $f^{-1}(V)$,  say $U$.  As $X$ and $Y$ are integral schemes of 
	finite type over a field,  we have 
	$\delta_X (a,b) = \delta _U (a,b)$ when $(a,b)$ is in $U$,  similarly for 
	$\delta_Y$ and $\delta_V$.Thus,  we can reduce to the case when both $X$ and $Y$ are 
	affine. 
	
	Let $X = \Spec B$ and $Y = \Spec
	A$.  Suppose $\psi : A \rar B$ is the corresponding ring homomorphism.  
	To show $\delta_X(a,b) \geq \delta_Y(f(a), f(b))$ when $a, b \in \Spec B$ is same as 
	showing $\delta_X(a,b) \geq \delta_Y(\psi ^{-1}(a), \psi ^{-1}(b))$.  
	The inequality holds if either of the following statements are true:
	\[(*) \hspace*{1mm} f_{a} : \Spec B_{a} \rar \Spec A _{\psi ^{-1}(a)} \text{ is surjective,} 
	\iff \psi \text{ 
	has going-down property.  }\]
	or,
	\[ (\prim*) \hspace*{1mm}  \bar{f} : \Spec (B/b) \rar \Spec (A/ \psi ^{-1}(b)) \text{ is surjective,} \iff \psi \text{ has 
	going-up property. } \]
	 
	 By our assumption $f$ is either flat or a closed immersion hence $f$ has either 
	 $(*)$ or $(\prim *)$.  This proves our claim.

\end{proof}

\begin{lemma}
\label{p and pf}

	Let $f: X \rar Y$ be as in Lemma \ref{dx and dy}.  If $p: Y \rar \Z$ 
	is a monotone  comonotone perversity on $Y$ then $p \circ f$ is a  monotone comonotone
	perversity on $X$.  
\end{lemma}

\begin{proof}
	It easy to see that if $p$ is a monotone perversity on $Y$,  then $p \circ f$ is a monotone 
	perversity on $X$.  Now,  for any monotone perversity $q$ on $X$,
	\begin{align*}
	q  \text{  is comonotone } 
	& \iff d_X - q \text { is monotone, }\\
	& \iff  d_X(y) - q(y) \geq d_X(x) - q(x) \text{ for all specialization pair } (y,x), \\
	& \iff d_X(y) - d_X(x) \geq  q(y)- q(x) \text{ for all specialization pair } (y,x), \\
	& \iff \delta_X (y,x) \geq  q(y)- q(x) \text{ for all specialization pair } (y,x).\\
	\end{align*}
	To show $p \circ f$ is a 
	comonotone perversity, we use the following sequence of inequalities,
	\[ 0 \leq p \circ f (a) - p \circ f(b) \leq \delta_Y(f(a), f(b)) \leq \delta_X (a, b) .\]
	The last inequality comes from Lemma \ref{dx and dy} and the middle one holds 
	as $p$ is comonotone.
\end{proof}

\begin{proposition}
\label{pullback}
	Let $f: X \rar Y$ be a flat morphism or a closed immersion of finite type integral schemes 
	over a field $k$.   Let $\cat A$ be a
	\tensor aisle of $\deriveb Y$, then $\mathbf{L}f^* \cat A$ is  a \tensor aisle of $\deriveb X$.
\end{proposition}

\begin{proof}
	By Theorem \ref{Perversity},  to show the \tensor preaisle $\mathbf{L}f^* \cat A$ is an aisle, it 
	is 
	equivalent to show that the corresponding perversity function 
	$p_{\mathbf{L}f^* \cat A}$ is comonotone.  Now,  Lemma \ref{pLf = pf} says 
	$p_{\mathbf{L}f^* \cat A} = p \circ f$.  And Lemma \ref{p and pf} says as $p \circ f$ is 
	comonotone whenever $p$ is comonotone.  This proves our claim.
\end{proof}

This shows tensor t-structures on $\deriveb X$ also restricts well to closed subschemes.

\begin{comment}

\subsection{Action of $\mathbf{Aut}(\deriveb X)$ on the set of tensor t-structures}

\begin{definition}
	Given a bounded Thomason filtration $\varphi$ on $X$ we define its length to be 
	\[\ell (\varphi) \colonequals min \{ j - i \mid \varphi (j) = \emptyset \text{ and }\varphi (i) = X \}.
	\]
	
	We define the \emph{length} of a tensor t-structure $(\cat U, \cat V)$ on $\deriveb X$ to be 
	the length of the corresponding unique Thomason filtration and denote by $\ell (\cat U)
	$.  Note that the length of the standard t-structures is one.  
\end{definition}
\end{comment}

\section{Weight structures on the bounded derived categories of curves}

In our earlier work \cite{DS23_2},  we observed that there are no non-trivial tensor weight 
structures on the bounded derived category of $\mathbb{P} ^1 _k$.  In this section,  we 
investigate the same question in the higher genus curves.   First, we recall the definition
of weight structures. 

\begin{definition} 
	A \emph{weight structure} on a triangulated category \cat T is a pair of 
	full subcategories $(\cat X,  \cat Y)$ satisfying the following properties:
	\begin{itemize}
		\item[w0.] \cat X and \cat Y are closed under direct summands.
		\item[w1.] $ \cat X [-1] \subset \cat X$ and $ \cat Y[1] \subset \cat Y$.
		\item[w2.] $\cat X \perp \cat Y$.
		\item[w3.] For any object $T\in \cat T$ there is a distinguished triangle
		\[ A \rar T \rar B \rar  A[1] \]
		where $A \in \cat X$ and $B \in \cat Y$.  The above triangle is called 
		a \emph{weight decomposition} of $T$.
	\end{itemize}
	
\end{definition}

Now, let $X$ be a smooth projective curve over a field $k$.  As we will be repeatedly using the 
non vanishing of $\Hom$-sets in $\deriveb X$ between coherent sheaves,  we 
collect them here in one place and refer to the following lemma whenever we use them.

\begin{lemma}
\label{Non-vanishing Lemma}
	Let $X$ be a smooth projective curve of genus $g$ over a field $k$.  Let $T$ be a 
	torsion sheaf of degree $d$ and $W$ be a locally free sheaf of rank $r$ on 
	$X$.  Then we have the following: \begin{center}
	\begin{tabular}{lll}
	$ (a) \dim_k \Hom (T,  W) = 0$, & \hspace{1cm} &  $(b)\dim_k \Hom (T,  W[1]) = r.d$, \\
	
	$(c)\dim_k \Hom (W, W) = r^2$, & \hspace{1cm} & $(d)\dim_k \Hom (W,  W[1]) = 
	r^2 .g$,\\
	
	$(e)\dim_k \Hom (W,  T) = r.d $, & \hspace{1cm} & $(f)\dim_k \Hom (W,  T[1]) = 0$.\\

	\end{tabular}		\end{center}
	
	\begin{comment}
	
	\begin{tabular}{lll}
	$ (a) \dim_k \Hom (T,  W) = 0$, & \hspace{1cm} &  $2. \dim_k \Hom (T,  W[1]) = r.d$, \\
	
	$3. \dim_k \Hom (W, W) = r^2$, & \hspace{1cm} & $4. \dim_k \Hom (W,  W[1]) = 
	r^2 .g$,\\
	
	$5. \dim_k \Hom (W,  T) = r.d $, & \hspace{1cm} & $6. \dim_k \Hom (W,  T[1]) = 0$.\\

	\end{tabular}

	\begin{itemize}
		\item[i.]  $\dim_k \Hom (T,  W) = 0$,
		\item[ii.] $\dim_k \Hom (T,  W[1]) = r.d$,
		\item[iii.] $\dim_k \Hom (W, W) = r^2$,
		\item[iv.] $\dim_k \Hom (W,  W[1]) = r^2 .g$.
	\end{itemize}
	\end{comment}
\end{lemma}

\begin{proof}
	Part $(a)$ and $(c)$ are obvious.  Now Part $(b)$. 
	\begin{align*}
	\Hom (T, W[1])^{\check{}}  
	& \cong \Ext ^1 (T,  W)^{\check{}}  \text{  \hspace*{2cm} ( by definition of Hom in \deriveb X )} \\
	& \cong \Hom (W, T \otimes \omega _X) \text{  \hspace*{1cm} ( by Serre Duality )}\\
	& \cong \Hom ( \CO _X , T \otimes \check{W}) \text{  \hspace*{9mm} ( by tensor 
	adjunction )} \\
	& \cong \Hom ( \CO _X ,  T ^{\oplus ^r })\\
	& \cong  \Gamma (X, T^{\oplus ^r }) \\
	\end{align*}
	
	And since $\dim_k \Gamma (X, T^{\oplus ^r }) = r.d$,  this proves part $(b)$.

	For part $(d)$,  
	\begin{align*}
	\Hom (W, W[1])^{\check{}} 
	& \cong \Ext ^1 (W,  W)^{\check{}}  \text{  \hspace*{2cm} ( by definition of Hom in \deriveb X )}\\
	& \cong \Hom (W, W \otimes \omega _X)  \text{  \hspace*{1cm} ( by Serre Duality )}\\
	& \cong \Hom ( W \otimes \check{W} ,  \omega _X )\text{  \hspace*{9mm} ( by tensor 
	adjunction )} \\
	& \cong \Hom ( \CO _X ^{\oplus ^ {r^2}} ,  \omega _X) \\
	& \cong \oplus ^{r^2} \Gamma (X , \omega _X)\\
	\end{align*}

	By assumption $\dim_k \Gamma (X,  \omega _X) =g$,  this proves part $(d)$.
	Part $(e)$ and $(f)$ can easily be obtained by applying Serre duality to $(b)$ and $(a)$
	respectively.
\end{proof}

\begin{lemma}
\label{delta lemma}
	Let $E$ be a complex in $\deriveb X$ such that $H ^{-1}(E)$ is a torsion sheaf.  Let 
	$V$ be a locally free sheaf on $X$ and $\delta : E \rar V[1]$ be a map 
	in $\deriveb X$.  Suppose the map $\delta$ fits in to the following distinguished triangle
	with the third term $F \in \deriveb X$, 
	
	\begin{center}
		\begin{tikzcd}
	 		 E \rar[" \delta "]  & V[1] \rar  & F \rar  & E[1].
		 \end{tikzcd}
		\end{center}

	Then,  for any torsion sheaf $T \in \mathrm{Coh} \hspace{1mm} X$,  we have 
	$\Hom (T , F) \neq 0$.
	
\end{lemma}

\begin{proof}
	As $\mathrm{Coh} \hspace{1mm} X$ has homological dimension one,  every complex in 
	$\deriveb X$ splits into direct sum of its cohomology sheaves.  Again since $\Ext ^i (-, -)$ 
	vanish for $i \geq 2$,  it is enough to assume that $E \cong H^{0}(E) \oplus H^{-1}(E)[1]$. 
	Now by the assumption $H^{-1}(E)$ is a torsion sheaf hence there are no non-zero map 
	from $H^{-1}(E)[1]$ to $V[1]$,  so the map $\delta$ only depends on  
	$\Hom (H^{0}(E), V[1])$.  
	
	By abuse of notation we denote the map $H^0(E) \rar V[1]$ by $\delta$; and for 
	simplicity we write $H^0(E)$ by $A$.  As we know $\Hom (A , V[1]) \cong \mathrm{Ext}^1 (A,  
	 V)$,  a morphism $\delta : A  \rar V[1]$ in $\deriveb X$ corresponds to an element of the 
	 group $\Ext ^1(A,V)$,  we denote the corresponding element in $\Ext ^1(A,V)$ again by 
	 $\delta$.
	
	Now we take the short exact sequence corresponding to $\delta \in \Ext^1(A, V)$, 
	say
		\begin{center}
		\begin{tikzcd}
	 		0 \rar & V \rar & B \rar  & A \rar & 0.
		 \end{tikzcd}
		\end{center} 
		
		From the short exact sequence it is clear that $B$ has a torsion-free part,  as $V$ is 
		locally free.   This short exact sequence gives rise to the following triangle,
		
		\begin{center}
		\begin{tikzcd}
	 		 V \rar & B \rar  & A \rar [" \delta "]  & V[1].
		 \end{tikzcd}
		\end{center} 
		
	Now, comparing the above triangle with the one in the hypothesis,  we can easily see
	that $F$ has a direct summand isomorphic to $B[1]$.  Thus to show  $\Hom (T,  F) \neq 0$,
	it is enough to show $\Hom (T ,  B[1]) \neq 0$.  Indeed this is the case by part $(b)$
	of Lemma \ref{Non-vanishing Lemma}

\end{proof}

\begin{lemma}
\label{V lemma}
	Let $X$ be a smooth projective curve over a field $k$ of genus $g \geq 1$, and
	 $(\cat X,  \cat Y)$ be a weight structure on $\deriveb  X$.  Suppose $\cat X \neq 0$.  
	Then,  for each $i \in \Z$ there is a locally free sheaf $W$ on $X$ such that $W[i] \in \cat X$.
\end{lemma}

\begin{proof}

	We define $\Lambda \colonequals \{ j \in \Z \mid \text{there is a locally free sheaf } W 
	\text{such that } W[j] \in \cat X \}$.  To prove the claim we need to show $\Lambda = \Z$.
	
	First we will prove that  $\Lambda$ is non-empty.  Suppose $\cat X$ contains only torsion 
	sheaves
	and their shifts. Then there is a torsion sheaf $T$ on $X$ such that $T[i] \in \cat X$,  without 
	loss of generality we can assume $i = 0$.  Let $V$ be a locally free sheaf on $X$.  By our 
	assumption $V[1] \notin \cat X$.   Now,  by part $(b)$ of Lemma \ref{Non-vanishing 
	Lemma} we have $ \Hom(T,  V[1]) \neq 0$,  therefore $V[1] \notin \cat Y$.  Note that 
	$V[1]$ can not have any weight decomposition by Lemma \ref{delta lemma}, which 
	contradicts the fact that$(\cat X,  \cat Y)$ is a weight structure.  Hence, $\cat X$ must contain 
	a locally free sheaf, atleast upto some shift,  this implies $\Lambda \neq \emptyset$.
	
	Since $\cat X$ is closed under negative shifts,  if $ j \in \Lambda$ then $j-1 \in \Lambda$.
	If $\Lambda \neq \Z$ then there is an integer $i$ such that $i+1 \notin \Lambda$.  This means 
	there is a locally free sheaf $W$ such that $W[i] \in \cat X$ and for any locally free sheaf
	$V$ we have $V[i+1] \notin \cat X$, in particular  $W[i+1] \notin \cat X$.   
	Again we can assume that $i = 0$.  Now by part $(d)$ of Lemma 
	\ref{Non-vanishing Lemma},  and the assumption that the genus $g \geq 1$, we have
	$\Hom (W,  W[1]) \neq 0$,  therefore,  $W[1] \notin \cat Y$.   
	Again,  by the Lemma \ref{delta lemma}, we get that $W[1]$ can not have a weight 
	decomposition contradicting the hypothesis. Therefore,  $\Lambda$ must be equal to $\Z$. 
	
\end{proof}

In \cite[Theorem 1.1]{Okawa},  Okawa has shown that there are no non-trivial semi-orthogonal 
decompositions on the bounded derived category of curves of genus greater than equal to one.  
We provide a different proof of the this fact using Lemma \ref{V lemma}. 

\begin{proposition}
\label{Sod theorem}
	Let $X$ be a smooth projective curve over a field $k$ of genus $g \geq 1$.  There are no 
	non-trivial semi-orthogonal decompositions on $\deriveb X$.
\end{proposition}

\begin{proof}	
	First note that, if $(\cat A,  \cat B)$ is a semi orthogonal decomposition on $\deriveb X$,  then 
	the pair $(\cat B, \cat A)$ is both a t-structure and a weight structure on $\deriveb X$.  Let 
	$x$ be a closed point of $X$ and $k(x)$ denote the degree one simple torsion sheaf
	supported on $x$.  Let us denote the set of all such torsion sheaves and their shifts by 
	$\cat S$,  that is, 
	\[\cat S \colonequals \{ k(x)[i] \mid \text{ where } x \in X \text { is a closed point and } 
	i \in \Z \}. \]

	If $\cat B =0$,  then $(\cat A, \cat B)$ is a trivial, so nothing 
	to prove.  Now if $\cat B \neq 0$,  we can apply Lemma \ref{V lemma} to $\cat B$.  Therefore
	for each integer $i$,  there is a locally free sheaf $V$ on $X$ such that $ V[i] \in \cat B$.  	
	We know that either $k(x)[i] \in \cat B$ or $k(x)[i] \in \cat A$.   Since $V[i] \in \cat B$,  by 
	part $(e)$ of Lemma \ref{Non-vanishing Lemma},  $k(x)[i]$ can not be in $\cat A$,  hence
	$ k(x)[i] \in \cat B$.  This implies $\cat S \subset \cat B$,  now by part $(b)$ of Lemma
	\ref{Non-vanishing Lemma}, there are no locally free sheaves or their shifts in $\cat A$,
	hence it must zero.  This proves that $(\cat A, \cat B)$ is trivial.

\end{proof}

Next, we specialize to the case of elliptic curves and show that not just semi-orthogonal 
decompositions, its bounded derived category admits no weight structures.  We borrow the 
basic facts of coherent sheaves on elliptic curves, and follow the notations and symbols from 
\cite{BB07}.   Let $E$ and $F$ be in $ \deriveb X$.  Recall the Euler form $\euler E F$ is defined 
as follows,
 \[\euler E F \colonequals \dim_k \Hom (E ,F) - \dim_k \Hom (E,  F[1]). \]   

\begin{lemma}
\label{Euler lemma}
	Let $(\cat X, \cat Y)$ be a weight structure on $\deriveb X$.  Let $E$ and $F$ be in 
	$\deriveb X$.  If $E \in \cat X$ and $\euler E F \neq 0$ then $F \notin \cat Y$.
\end{lemma}

\begin{proof}
	This is immediate from the definition of weight structures and the Euler form.
\end{proof}

\begin{lemma}
\label{Slope lemma}
	Let $X$ be an elliptic curve over $k$.  Let $V$ and $W$ be two locally
	free sheaves having different slopes.  Then, $\euler V W \neq 0$. 
\end{lemma}

\begin{proof}
	By \cite[Theorem 4.13]{BB07},   we have $\euler V W = \chi(W)rk(V) - \chi (V)rk(W)$.  Now 
	since the slopes of $V$ and $W$ are different, that is,  $\chi (V)/rk(V ) \neq 
	\chi (W)/ rk(W)$ we have $\euler V W \neq 0$.
\end{proof}

\begin{theorem}
\label{Elliptic Curve Theorem}
	Let $X$ be an elliptic curve.  There are no non-trivial weight structures on $\deriveb X$.
\end{theorem}

\begin{proof}
	Suppose $(\cat X,  \cat Y)$ be a weight structure on $\deriveb X$.  Suppose 
	$\cat X \neq 0$.  We will show that $\cat Y = 0$.For any torsion sheaf $T$ on $X$ and for 
	any $i \in \Z$.  We will show $T[i] \notin \cat Y$.  
	Indeed,  by Lemma \ref{V lemma}, there is a locally free sheaf $W$ such that 
	$W[i] \in \cat X$.  Now,  note that by Lemma \ref{Non-vanishing Lemma},  
	$\euleri W T \neq 0$.  Hence,  by Lemma \ref{Euler lemma}, we have $T[i] \notin \cat Y$.
	
	In a similar fashion we will show that $\cat Y$ contains no locally free sheaves or their 
	shifts.  Let $V$ be a locally free sheaf on $X$ and $i$ be
	an integer.  We will show $V[i] \notin \cat Y$.  Now by Lemma \ref{V lemma}, there is a 
	locally free sheaf $W$ such that $W[i] \in \cat X$.  We can assume $i = 0$ without loss
	of generality. If $\mu(V) = \mu(W)$,  then $\Hom (W, V) \neq 0$,
	therefore $V \notin \cat Y$.  If $\mu(V) \neq \mu(W)$,  we can apply Lemma \ref{Slope 
	lemma} to get $\euler V W \neq 0$.  Now by Lemma \ref{Euler lemma} we get $V \notin \cat Y
	$.
	
\end{proof}

Next, we will recall the notion of tensor weight structures on $\deriveb X$.  With the technical 
result of Lemma \ref{V lemma},  and the tensor condition we can prove that unlike the 
tensor t-structures,  there are only trivial tensor weight structure on $\deriveb X$. 

\begin{definition}
	We say a weight structure $(\cat X, \cat Y)$ on $\deriveb X$ is a tensor weight structure if  
	for any $E \in \perfectzero X$ and any $ F \in \cat X$ the derived internal sheaf hom 
	$\mathbf{R}  \cat H om ^{\cdot} (E,  F)$ is in $\cat X$.  
\end{definition}  

\begin{remark}
	In \cite[Definition 13]{DS23_2} we had called such a weight structure 
	\enquote{\emph{$\otimes ^c$-weight structure}} and thought of it as a weaker notion.   The 
	notion of tensor weight structure on a triangulated category were defined using a natural 
	choice of a preaisle of the ambient category,  see Definition 7 of loc. 
	cit.  In our case it would have been the aisle $\mathbf{D}^ {b,  \leq 0} (X)$ of the 
	standard t-structure.  However, in view of Definition \ref{tensor definition}, Remark 
	\ref{key remark},  and the lack of symmetry in Theorem 15 of \cite{DS23_2} it now seems 
	that the above definition using $\perfectzero X$ is the correct one.
\end{remark}

\begin{proposition}
\label{weight structure theorem}
	Let $X$ be a smooth projective curve over a field $k$ of genus $g \geq 1$.  There are no 
	non-trivial tensor weight structures on $\deriveb X$.
\end{proposition}

\begin{proof}
	Suppose $(\cat X,  \cat Y)$ is a tensor weight structure on $\deriveb X$ and $\cat X \neq 
	0$.  Following the same argument as in  the proof of Theorem \ref{Elliptic Curve Theorem},  
	we can deduce that $\cat Y$ contains no torsion sheaf or its shifts.  Next,  we will show 
	$\cat Y$ contains no locally free sheaf or its shifts,  and hence $\cat  Y$ must be zero.
	
	Suppose $V$ is a locally free sheaf and $V [i] \in \cat Y$.  By Lemma \ref{V lemma}, there is a 
	locally free sheaf $W$ such that $W[i] \in \cat X$.  Suppose the rank of $W$ is say $r$.  We 
	then have,
	
	\begin{align*}
		\mathbf{R}  \cat H om ^{\cdot} (W \otimes V,  W[i])
		& \cong \mathbf{R}  \cat H om ^{\cdot} (V,  \check{W} \otimes W [i] )\\
		& \cong \mathbf{R}  \cat H om ^{\cdot} (V,  \CO _X ^{\oplus ^{r^2}} [i] )\\
		& \cong \oplus ^{r^2} \mathbf{R}  \cat H om ^{\cdot} (V,  \CO _X )[i]\\
		& \cong \oplus ^{r^2} V[i]
	\end{align*}
	
	Now by the tensor property of $\cat X$ and that it is closed under direct summands, we get
	that $V[i] \in \cat X$.   Since $\cat X \cap \cat Y = 0$,  so $V = 0$.  This proves that $\cat Y$
	is trivial.  
	
\end{proof}

%\newpage

\section*{Acknowledgement}

The author expresses gratitude for the serene and beautiful seaside of TIFR,  Mumbai and 
extends sincere thanks to the support staff for their invaluable assistance.  The questions 
addressed in this article were conceived during my PhD years,  and I am deeply grateful to my 
advisor, Umesh Dubey, for encouraging me to pursue these questions.

\bibliographystyle{alpha}
\bibliography{Project3}

\end{document}